\documentclass[11pt]{amsart}

\title{On odd-dimensional Modular Tensor Categories}

\usepackage{graphicx}
\usepackage{amsmath,amsthm, amscd, amssymb, amsfonts}
\usepackage{enumitem}
\usepackage{tikz-cd}
\usepackage{color}
\usepackage[margin=1.25in]{geometry}
\usepackage{nicefrac}
\numberwithin{equation}{section}\theoremstyle{plain}
\usepackage{amsthm}

 \newtheorem*{rep@theorem}{\rep@title}
\newcommand{\newreptheorem}[2]{%
\newenvironment{rep#1}[1]{%
 \def\rep@title{#2 \ref{##1}}%
 \begin{rep@theorem}}%
 {\end{rep@theorem}}}
\makeatother

\newreptheorem{theorem}{Theorem}

\newtheorem{theorem}{Theorem}[section]

\newtheorem{lemma}[theorem]{Lemma}
\newtheorem{cor}[theorem]{Corollary}

\newtheorem{conjecture}[theorem]{Conjecture}
\newtheorem{pro}[theorem]{Proposition}

\newtheorem{remark}[theorem]{Remark}

\newcommand\pf{\begin{proof}}
\newcommand\epf{\end{proof}}

\newcommand\Hom{\operatorname{Hom}}
 
\newcommand\End{\operatorname{End}}

\newcommand\ev{\operatorname{ev}}
\newcommand\coev{\operatorname{coev}}
\newcommand\Rep{\operatorname{Rep}}

\newcommand\FPdim{\operatorname{FPdim}}

\newcommand\cd{\operatorname{c.d.}}
\newcommand\Tr{\operatorname{Tr}}
\newcommand\Id{\operatorname{Id}}	

	\newcommand\VecGw{\operatorname{\textbf{Vec}_G^{\omega}}}

\newcommand\id{\operatorname{id}}

\newcommand\modulo{\operatorname{mod}}
\newcommand\di{\operatorname{d}}
\newcommand\rank{\operatorname{rank}}

\begin{document}

\author[A. Czenky]{Agustina Czenky}
\address{Department of Mathematics, University of Oregon}
\email{aczenky@uoregon.edu}

\author[J. Plavnik]{Julia Plavnik}
\address{Department of Mathematics, Indiana University}
\email{jplavnik@iu.edu}

\begin{abstract}
    We study odd-dimensional modular tensor categories and maximally non-self dual (MNSD) modular tensor categories of low rank. We give lower bounds for the ranks of modular tensor categories in terms of the rank of the adjoint subcategory and the order of the group of invertible objects. As an application of these results, we prove that odd-dimensional modular tensor categories of ranks 13 and 15 are pointed. In addition, we show that odd-dimensional tensor categories of ranks 19, 21 and 23 are either pointed or perfect. 
\end{abstract}

\maketitle
\section{Introduction}

Modular tensor categories (MTCs) with integral Frobenius-Perron dimension have been deeply studied in the last decade, see for example \cite{BR}, \cite{BGHKNNPR}, \cite{DGNO1}, \cite{DLD}, \cite{DN}, \cite{DT}, \cite{EGO}, \cite{ENO1}, \cite{ENO2}, \cite{NR}. One large class of examples is given by odd-dimensional MTCs ~\cite{GN}. Ng and Schauenburg proved that if the dimension of a MTC is odd then the category is maximally non-self dual (MNSD), i.e.  the only self-dual simple object is the unit object \cite[Corollary 8.2]{NS}.  In the other direction, Hong and Rowell 
showed in \cite[Theorem 2.2]{HR} that MNSD MTCs are always integral, and as a consequence they must be odd-dimensional. In this work we will focus our attention on low rank MNSD MTCs, or equivalently, low rank odd-dimensional MTCs.

In~\cite{BR}, the authors studied odd-dimensional MTCs in more detail. In particular, they asked if odd-dimensional MTCs are neccesarily group-theoretical. A negative answer to this question can be deduced from results of Larson and Jordan in~\cite{JL}. They constructed explicit examples of non group-theoretical integral fusion categories of dimension $pq^2$ for the case where $p$ is an odd prime that divides $q+1$ \cite[Theorem 1.1]{JL}. Hence there are non-group-theoretical integral modular categories obtained as the Drinfeld centers of said categories.

Bruillard and Rowell showed in~\cite{BR} that any MNSD MTC of rank at most 11 is always pointed. Recall that a fusion category is called pointed if all its simple objects are invertible, or, equivalently, all its simple objects have Frobenius-Perron dimension equal to 1. They also found an example of a MNSD MTC of rank $25$ that is not pointed (but it is group-theoretical). A natural follow-up question is if there exists a non-pointed MNSD MTC of rank less than $25$ \cite{BR}.

In this manuscript we continue the study of odd-dimensional MTCs. We prove some useful relations between the ranks of the components of a faithful grading of an odd-dimensional MTC. Moreover, we give some bounds of the rank of the category in terms of the order of the group of invertibles, the rank of the adjoint subcategory, and the primes dividing the order of the group of invertibles of the adjoint. We apply these general results to classify odd-dimensional MTCs of ranks $13$ and $15$. More precisely, we show that all simple objects are invertible, thus the category can be understood in group-theoretical terms. We prove 
that if a odd-dimensional MTC of rank between $19$ and $23$  is not pointed, then it must be \emph{perfect}, which means that the unit is the unique invertible object. This is summarized in the following theorem.

\begin{reptheorem}{thm:MNSD}
Let $\mathcal C$ be an odd-dimensional MTC.
\begin{enumerate}[leftmargin=*, label=(\alph*)]
\item If $\rank(\mathcal C)=13$ or $15$, then $\mathcal C$ is pointed.
\item\label{item:MNSD-rank<25} If $\rank(\mathcal C)=17$, then $\mathcal C$ is either pointed, perfect, or has 3 invertible objects. Moreover, if the latter case exists, then $\mathcal C$ is equivalent to a categorification of the ring $R_{3,H}$  as defined in \cite[Definition 1.3]{JL}, where $H$ is a finite abelian group of order 25.
\item If $\rank(\mathcal C)=19, 21$ or $23$, then $\mathcal C$ is either pointed or perfect.
\end{enumerate}
\end{reptheorem} 

The potential rank 17 non-pointed, non-perfect modular category was missing in a previous version of this article. Palcoux found this potential example and the possible dimensions were described in \cite[Theorem 1.9]{ABPP}.


Hence, for an odd-dimensional MTC $\mathcal C$ of rank between $19$ and $23$, the fusion subcategory $\mathcal C_{pt}$ generated by invertible objects is confined to the two ends of the spectrum: either $\mathcal C_{pt}=\operatorname{Vec}$ or $\mathcal C_{pt}=\mathcal C$.

There is a major conjecture in fusion categories that states that every weakly integral fusion category is weakly group-theoretical~\cite[Question 2]{ENO2}.
As a consequence of our results, a non-pointed odd-dimensional MTC of rank between 19 and 23 cannot be weakly group-theoretical. This follows since for weakly group-theoretical fusion categories there is a version of the Feit-Thompson theorem: if a weakly group-theoretical fusion category is odd-dimensional then it is solvable~\cite[Proposition 7.1]{NP}. It is known that solvable fusion categories contain non-trivial invertible objects~\cite[Proposition 4.5]{ENO2}.  The veracity of this conjecture would imply that odd-dimensional MTCs of ranks between 19 and 23 are pointed.  In fact, any perfect odd-dimensional MTC would yield a counter-example for said conjecture. This leads us to the following conjecture: 
\begin{conjecture}\label{conjecture: odd dim non perfect}
Odd-dimensional MTCs have at least one non-trivial invertible object, i.e. they cannot be perfect.
\end{conjecture}

By the argument explained above using Feit-Thompson for weakly group-theoretical fusion categories, this conjecture would be an immediate consequence of the weakly integral conjecture.  


Another question is if it is possible to remove the hypothesis of weakly group-theoretical in the fusion categorical version of Feit-Thompson's Theorem:

\begin{conjecture}\label{Conjecture:odd modular are solvable}
Odd-dimensional MTCs are solvable.
\end{conjecture}
It follows from \cite[Proposition 4.5]{ENO2} that Conjecture \ref{Conjecture:odd modular are solvable} is equivalent to the stronger statement that odd-dimensional fusion categories are solvable. In fact, if $\mathcal C$ is an odd-dimensional fusion category, then its Drinfeld center $\mathcal {Z(C)}$ is an odd-dimensional MTC. Assuming Conjecture \ref{Conjecture:odd modular are solvable} this would imply that $\mathcal{Z(C)}$ is solvable, and since it is Morita equivalent to $\mathcal C \boxtimes \mathcal C^{\operatorname{op}}$ we conclude that $\mathcal C$ is also solvable \cite[Proposition 4.5]{ENO2}.  

Note that the above conjectures are equivalent as follows\footnote{The equivalence of these statements was pointed out to us by C. Galindo.}. Assuming Conjecture \ref{conjecture: odd dim non perfect}, we prove Conjecture \ref{Conjecture:odd modular are solvable} by induction on $\dim(\mathcal C).$ Let $g$ be a non-trivial invertible object in $\mathcal C$; we may assume that the order of $g$ is prime. 
Consider the fusion subcategory $\mathcal{C}\langle g\rangle$. If $\mathcal{C}\langle g\rangle$ is not Tannakian then it must be modular and $\mathcal C= \mathcal C \boxtimes \mathcal{C}\langle g\rangle'$. Since $\mathcal{C}\langle g\rangle'$ is a MTC of dimension strictly less than $\dim(\mathcal C)$ then by induction it is solvable and thus $\mathcal C$ is solvable \cite[Proposition 4.5]{ENO2}.  On the other hand, if $\mathcal{C}\langle g\rangle$ is Tannakian then $\mathcal{C}\langle g\rangle$ is equivalent to $\Rep(G)$ for $G=\langle g\rangle$, and so the trivial component $(\mathcal C_G)_0$ of the de-equivariantization $\mathcal C_G$ of $\mathcal C$ is a MTC of dimension strictly less than $\dim(\mathcal C)$ and hence solvable by induction. It follows that $\mathcal C$ is also solvable by \cite[Proposition 4.5]{ENO2}. On the other hand, Conjecture \ref{conjecture: odd dim non perfect} follows from Conjecture \ref{Conjecture:odd modular are solvable} and \cite[Proposition 4.5]{ENO2}. 

By an analogous reasoning, we have that Conjecture $\ref{conjecture: odd dim non perfect}$ is equivalent to \cite[Question 2]{ENO2} for the odd-dimensional case, namely, odd-dimensional weakly integral fusion categories are weakly group theoretical. 

\medbreak 
The paper is organized as follows. In Section~\ref{section: preliminaries} we introduce the basic notions and we prove some results that we will use throughout this article. In Section~\ref{section: cd_p} we study MTCs whose simple objects have dimension a multiple of an odd prime number $p$ and MTCs with dimension a power of $p$. It is known that MTCs of dimension $p^n$ for $n\leq 4$ are pointed, see for example \cite{DN} for the result when $n=4$. This is not longer true for $n=5$, 
we show that they must be either pointed or a $\mathbb Z_p\times \mathbb Z_p$-gauging of a pointed modular category of dimension $p$. In Section~\ref{section: perfect} we show that odd-dimensional MTCs with exactly one invertible object (the unit) have no non-trivial symmetric subcategories. This fact has strong consequences: every fusion subcategory of such a category is modular, and therefore the category is split. In Section~\ref{section: ranks} we find bounds for the rank of a category in terms of data associated to its universal grading. We also give conditions on the rank of the graded components in relation to the fixed points by the action of the group of invertible objects of the category. Lastly, in Section~\ref{section: mnsd low rank}
 we apply the results on the previous sections to classify odd-dimensional MTCs of rank between $13$ and $23$.


\section*{Acknowledgments}
The authors specially thank A. Brugui\`eres for enlightening discussions. They also thank P. Bruillard, E. Campagnolo, H. Pe\~na Polastri, G. Sanmarco and L. Villagra for useful comments. The authors thank C. Galindo, M. , M\"uller, V. Ostrik, E. Rowell, and A. Schopieray for helpful remarks on an early draft.

Substantial portions of this project were discussed at the CIMPA Research Schools on Quantum Symmetries (Bogot\'a, 2019) and Hopf Algebras and Tensor Categories (C\'ordoba, 2019) and the authors would like to thank all the various organizers and granting agencies involved.
Both authors were partially supported by NSF grant DMS-0932078,
administered by the Mathematical Sciences Research Institute while the first author attended a workshop and the second author was in residence
at MSRI during the Quantum Symmetries program in Spring 2020.
JP was partially supported by NSF grants DMS-1802503 and DMS-1917319. JP gratefully acknowledges the support of Indiana University, Bloomington, through a Provost's Travel Award for Women in Science.
\section{Preliminaries}\label{section: preliminaries}
In this paper, we will always work over an algebraically closed field $\textbf{k}$ of characteristic zero. 
We refer to~\cite{BK} and~\cite{EGNO}
for the basic theory of fusion categories and braided fusion categories, and for terminology used throughout this paper. 

\subsection{Fusion categories}
A fusion category $\mathcal C$ over $\textbf{k}$ is a $\textbf{k}$-linear
semisimple rigid tensor category with a finite number of simple objects
and finite dimensional spaces of morphisms, and such that the endomorphism algebra
of the identity object (with respect to the tensor product) $\textbf{1}$ is $\textbf{k}$. 

Let $\mathcal C$ be a fusion category. We shall denote by $\mathcal{O(C)}$ the set of isomorphism classes of simple objects of $\mathcal C$. After fixing an enumeration of $\mathcal{O(C)}$, we denote the fusion coefficients by $N^k_{ij} := \dim \Hom(X_i \otimes X_j
, X_k)$, where $X_i, X_i, X_k \in \mathcal{O(C)}$.
The Frobenius-Perron dimension of an object $X$ is denoted by $\FPdim{X}$ (see~\cite{ENO1}). We will use the notation $\cd(\mathcal C)$ to indicate the set  $$\cd(\mathcal C)=\{\FPdim(X) : X \in \mathcal{O(C)}\},$$ and the positive real numbers $\FPdim X$, $X \in \mathcal{O(C)}$, will be called the \emph{irreducible degrees} of  $\mathcal C$. The Frobenius-Perron dimension of $\mathcal C$ will be denoted by $\FPdim(\mathcal C)$. We say that a $\mathcal C$ is \emph{weakly integral} if $\FPdim(\mathcal C)$ is an integer, and \emph{integral} if $\FPdim(X)$ is an integer for all simple $X$. 

Given an object $X$ and an isomorphism $f:X\to X^{**}$, we can define a (left) \emph{trace} given by the composition
	\begin{equation}
	\begin{tikzcd}
	\Tr_X(f):\textbf{1} \arrow{r}{\coev_X} &X\otimes X^* \arrow{r}{f\otimes \id_{X^*}} 
	&X^{**}\otimes X^* \arrow{r}{\ev_{X^*}} &\textbf{1}.
	\end{tikzcd}
	\end{equation}

A \emph{pivotal structure} on $\mathcal C$ is a natural isomorphism $\psi: \text{Id} \to (-)^{**}$, i.e., an isomorphism between the double dual and identity functors. Associated to a pivotal structure $\psi$ we get the notion of \emph{quantum dimension} $\di_X$ of an object $X\in \mathcal C$ as the scalar given by 
	\begin{equation}
	\begin{tikzcd}
d_X:=\Tr_X(\psi_X)\in \End(1)\simeq \textbf k.
	\end{tikzcd}
	\end{equation}
	The \emph{quantum dimension} of $\mathcal C$ is given by $\dim(\mathcal C)=\sum\limits_{X\in \mathcal{O(C)}} \di_X \di_{X^*}.$  A pivotal structure on a category $\mathcal C$ is called \emph{spherical} if it satisfies the property $\di_X=\di_{X^*}$ for all simple object $X\in \mathcal C$. Note that in this case, $\dim(\mathcal C)= \sum\limits_{X\in \mathcal{O(C)}} \di_X^2$.
	
A pivotal fusion category $\mathcal{C}$ is said to be \emph{pseudo-unitary} if $\dim(\mathcal{C})=\FPdim(\mathcal{C})$.
This happens to be equivalent to $\di_X^2=\FPdim(X)^2$ for all $X \in \mathcal{O(C)}$. 
Any pseudo-unitary fusion category admits a unique spherical structure, with respect to which the categorical dimensions of all simple objects are positive, and coincide with their Frobenius-Perron dimensions \cite[Proposition 8.23]{ENO1}.
By \cite[Proposition 9.6.5]{EGNO}, weakly integral fusion categories are pseudo-unitary. We will sometimes use the Frobenius-Perron dimension and the global dimension of $\mathcal C$ indifferently when working with pseudo-unitary categories \cite[Corollary 9.6.6]{EGNO}.

A fusion category
is \emph{pointed} if all simple objects are invertible, which is equivalent to having $\cd(\mathcal C)=\{1\}$. 
In this case $\mathcal C$ is equivalent to the category of finite dimensional $G$-graded vector
spaces $\VecGw$, where $G$ is a finite group and $\omega$ is a 3-cocycle on $G$ with coefficients in $\textbf{k}^{\times}$ codifying the associativity structure. 
The group of isomorphism classes of invertible objects of $\mathcal C$ will be indicated by $\mathcal{G(C)}$. 
The largest pointed subcategory of $\mathcal C$ will be denoted by $ \mathcal{C}_\mathrm{pt}$, that is the fusion subcategory of $\mathcal{C}$ generated by $\mathcal{G(C)}$. We will often identify the objects in $\mathcal{G}(\mathcal{C})$ with the invertible objects in $\mathcal{C}_\mathrm{pt}$. A fusion category is \emph{group-theoretical} if it is Morita equivalent to a pointed fusion category.

\subsubsection{The universal grading}

Let $G$ be a finite group. A $G$-grading on a fusion category $\mathcal C$ is a decomposition $\mathcal C = \displaystyle\bigoplus_{g\in G} \mathcal C_g$, such that $\otimes: \mathcal C_g\times \mathcal C_h \to \mathcal C_{gh}$, $\mathbf{1}\in \mathcal C_{e}$, and $*:\mathcal C_g \to \mathcal{C}_{g^{-1}}$. Such grading is said to be \emph{faithful} if $\mathcal C_g\ne 0$ for all $g \in G$. In this case, we say that $\mathcal C$ is a $G$-extension of the trivial component $\mathcal C_e$.

If $\mathcal C$ is a fusion category endowed with a faithful grading $\mathcal C = \displaystyle\bigoplus_{g\in G} \mathcal C_g$, then all the components $\mathcal C_g$ have the same Frobenius-Perron dimension \cite[Proposition 8.20]{ENO2}. Hence, $\FPdim(\mathcal C)=|G|\FPdim(\mathcal C_e)$. By \cite{GN}, any fusion category $\mathcal{C}$ admits a canonical faithful grading $\mathcal{C}= \oplus_{g \in U(\mathcal{C})} \mathcal{C}_g$,  called the \emph{universal grading}; its trivial component coincides with the so-called \emph{adjoint subcategory} $\mathcal{C}_{ad}$, which is defined as the fusion subcategory generated by $\{ X\otimes X^* \colon X \in \mathcal{O(C)} \}$. If $\mathcal{C}$ is equipped with a braiding, then $\mathcal{U(C)}$ is abelian. Moreover, if $\mathcal C$ is modular then $\mathcal{U(C)}$  is isomorphic to the group of (isomorphism classes of) invertibles $\mathcal{G(C)}$ \cite[Theorem 6.3]{GN}. 

\subsubsection{Nilpotent, solvable and weakly group-theoretical fusion categories}

Let $\mathcal{C}$ be a fusion category. The \emph{upper central series} of  $\mathcal{C}$ is the sequence of fusion subcategories of $\mathcal{C}$ defined recursively by
$$\mathcal{C}^{(0)}=\mathcal{C}   \text { and }\mathcal{C}^{(n)}=\big(\mathcal{C}^{(n-1)}\big)_{ad} 
\ \text{ for all } n\geq 1.$$

The fusion category $\mathcal{C}$ is \emph{nilpotent} if its upper central series converges to the trivial fusion subcategory $\operatorname{Vec}$ of finite dimensional vector spaces; that is, there exists $n\in \mathbb N$ such that $\mathcal{C}^{(n)}=\operatorname{Vec}$ (see \cite{GN}, \cite{ENO1}). Alternatively, a fusion category is nilpotent if there is a sequence of fusion subcategories $\mathcal C_0=\operatorname{Vec} \subset \dots \subset \mathcal C_n=\mathcal C$ and finite groups $G_1, \dots, G_n$ such that $\mathcal C_i$ is a $G_i$-extension of $\mathcal C_{i-1}$ for all $i$. If, moreover, the groups $G_i$ are cyclic the category is said to be \emph{cyclically nilpotent}.
As is the case for finite groups, fusion categories of Frobenius-Perron dimension a power of a prime integer are known to be nilpotent by \cite[Theorem 8.28]{ENO1}. Also, when $\mathcal C$ is a nilpotent fusion category $\FPdim(X)^2$ divides $\FPdim(\mathcal C_{\mathrm{ad}})$ for all $X\in \mathcal{O(C)}$ \cite{GN}. We make repeated use of these result. 

A fusion category $\mathcal{C}$ is \emph{weakly group-theoretical}, respectively \emph{solvable}, if it is Morita equivalent to a nilpotent fusion category, respectively to a cyclically nilpotent fusion category \cite{ENO1}. Both of these Morita classes are closed under Deligne tensor product, Drinfeld centers and fusion subcategories. The Morita class of weakly group-theoretical fusion categories is also closed under group extensions and equivariantizations. On the other hand, the class of solvable categories is closed under extensions and equivariantizations by \emph{solvable} groups, and also by taking component categories of quotient categories. See \cite[Propositions 4.1, 4.5]{ENO1}.

More can be said about categories that come equipped with a braiding. Namely, for a braided solvable fusion category $\mathcal C$, we have that either $\mathcal C=\text{Vec}$ or $\mathcal {G(C)}$ is not trivial \cite[Proposition 4.5]{ENO2}. It is also known that braided nilpotent fusion categories are solvable \cite[Proposition 4.5]{ENO1}.

\subsection{Modular tensor categories}

Let $\mathcal{C}$ be a braided fusion category with braiding $\sigma$. A \emph{twist} on $\mathcal{C}$ is a  natural isomorphism $\theta:\Id_\mathcal{C} \to \Id_\mathcal{C}$ such that
	\begin{align}
	\theta_{X\otimes Y}=(\theta_X \otimes \theta_Y) \circ \sigma_{Y,X} \circ \sigma_{X,Y},
	\end{align}
	for all $X,Y \in \mathcal{C}$. A twist is called a \emph{ribbon structure} if $(\theta_X)^*=\theta_{X^*}$ for all $X \in \mathcal{C}$.

A \emph{pre-modular category} is a braided fusion category endowed with a compatible ribbon structure. Equivalently, a pre-modular tensor category is a braided fusion category equipped with a spherical structure. This follows since when we have a pivotal structure on $\mathcal C$, we can define a twist by $\theta_X = (\Tr_X\otimes \Id_X)\circ \sigma_{X,X}$ \cite{Br}. Moreover, if we started with a spherical structure, this twist is a ribbon structure.

Let $\mathcal C$ be a premodular category with twist $\theta$. A fusion subcategory $\mathcal D$ of $\mathcal C$ is \emph{isotropic} if $\theta$ restricts to the identity on  $\mathcal D$, i.e., if $\theta_X = \text{Id}_X$ for all $X \in \mathcal D$.

Let $\mathcal C$ be a premodular tensor category, with braiding $\sigma_{X,Y}:X\otimes Y\xrightarrow{\sim}Y\otimes X$. The \emph{S-matrix} $S$ of $\mathcal C$ is defined by $S:= \left(s_{X,Y}\right)_{X,Y \in \mathcal{O(\mathcal{C})}}$, where $s_{X,Y}=\Tr(\sigma_{Y,X}\sigma_{X,Y})$. We can obtain the entries of the $S$-matrix in terms of the twists, fusion rules, and quantum dimensions via the so-called \emph{balancing equation}\label{balancing}
	\begin{align}
	s_{X,Y}=\theta_X^{-1}\theta_Y^{-1} \sum\limits_{Z \in \mathcal{O(\mathcal{C})}} N_{XY}^Z \theta_Z \di_Z,
	\end{align}
for all $X,Y \in \mathcal{O(\mathcal{C})}$ \cite[Proposition 8.13.7]{EGNO}.	

A premodular tensor category $\mathcal C$ is said to be \emph{modular} if the $S$-matrix $S$ is non-degenerate.

\subsubsection{Centralizers in braided fusion categories}
Let $\mathcal{C}$ be a braided fusion category and let $\mathcal{K}$ be a fusion subcategory of $\mathcal{C}$. The \emph{M\"{u}ger centralizer} of $\mathcal{K}$ is the fusion subcategory $\mathcal K'$ of $\mathcal C$ consisting of all objects $Y$ in $\mathcal{C}$ such that
\begin{align}\label{centralizador}
\sigma_{Y,X}\sigma_{X,Y}=\id_{X\otimes Y}, \ \ \text{for all}\ X \in \ \mathcal{K} \ \ \ \ \ \ \text{\cite{Mu}}.
\end{align}  

If $\mathcal{C}$ is modular then $\mathcal{K}=\mathcal{K}''$ and $\dim(\mathcal{K})\dim(\mathcal{K}')=\dim(\mathcal{C})$ \cite[Theorem 3.2]{Mu}; moreover, $\mathcal{C}_\mathrm{pt}=\mathcal{C}_{\mathrm{ad}}'$ and  $\mathcal{C}_\mathrm{pt}'=\mathcal{C}_{\mathrm{ad}}$ \cite[Corollary 6.9]{GN}. 

A necessary and sufficient condition for $\mathcal{K}$ to be modular is $\mathcal{K} \cap \mathcal{K}' = \ \text{Vec}$. This tells us that $\mathcal C$ is modular if and only if $\mathcal C'=\text{Vec}.$ Moreover, if $\mathcal K$ is a modular subcategory of $\mathcal C$ then $\mathcal{K}'$ is also modular and $\mathcal{C} \simeq \mathcal{K} \boxtimes \mathcal{K}'$ as braided fusion categories \cite{Mu}. On the other hand, we say a braided fusion category $\mathcal C$ is \emph{symmetric} if it is equal to its centralizer, that is, $\mathcal C= \mathcal C'$. Odd-dimensional symmetric fusion categories are isotropic \cite[Corollary 2.7]{DGNO1}.
\begin{remark} Let $\mathcal{C}$ be a MTC and $\mathcal K$ be a fusion subcategory of $\mathcal C$. Then $\mathcal K \cap \mathcal K'$ is symmetric. 
\end{remark}

Let $\mathcal K$ be a fusion subcategory of a fusion category $\mathcal C$. The \emph{commutator} of $\mathcal K$ is the fusion subcategory $\mathcal K^{\mathrm{co}}$ is generated by all simple objects $X \in \mathcal C$ such that $X \otimes X^* \in \mathcal K$.

\begin{remark}\label{Cad_pt_symm}
If $\mathcal{C}$ is a braided fusion category, then $(\mathcal{C}_{\mathrm{ad}})_{\mathrm{pt}}$  is symmetric. 
\end{remark}
\begin{proof}

First note that $(\mathcal{C}_{\mathrm{ad}})'=(\mathcal{C}')^{co}$ \cite[Proposition 3.25]{DGNO2}, i.e, $X\in (\mathcal{C}_{\mathrm{ad}})'$ if and only if $X\otimes X^* \in \mathcal{C}'$. In particular, $\mathcal{C}_\mathrm{pt}\subseteq  (\mathcal{C}_{\mathrm{ad}})' $, and the claim follows since  $(\mathcal{C}_{\mathrm{ad}})_{\mathrm{pt}}  \subseteq \mathcal{C}_\mathrm{pt} \subseteq  (\mathcal{C}_{\mathrm{ad}})' \subseteq ((\mathcal{C}_{\mathrm{ad}})_{\mathrm{pt}})'.$
\end{proof}
\bigbreak


\subsubsection{Actions of $\mathcal{G(C)}$ on $\mathcal{O(C)}$ in MTCs.}

From now on let $\mathcal{C}$ be a pre-modular fusion category.

\begin{lemma}\label{gX=X}
Let $g$ be an invertible object in $\mathcal{C}_{\mathrm{ad}}$ such that $\theta_g=1$. Suppose $g\otimes X=X$ for all non-invertible simple $X\not\in \mathcal{C}_{\mathrm{ad}}$. Then the rows of the S-matrix corresponding to (the isomorphism classes of) $g$ and \textbf{1} are equal.
\end{lemma}

\begin{proof}
Let $g$ be as above. The balancing equation \eqref{balancing}  yields
\begin{equation}\label{G2}
    s_{g,X}= \theta_g^{-1}\theta_X^{-1}\di_X\theta_{X}= \di_X, \text{ for all non-invertible simple $X\not\in \mathcal C_{\mathrm{ad}}$},
\end{equation}
that is, for all simple $X$ such that $X\not\in \mathcal {C}_{\mathrm{ad}}\cup \mathcal{C}_\mathrm{pt}.$

Now we compute $s_{g,X}$ for $X\in \mathcal{C}_\mathrm{pt}\cup \mathcal {C}_{\mathrm{ad}}$.  Assume first that $X\in \mathcal C_{\mathrm{ad}}$. Since $g\in (\mathcal{C}_\mathrm{pt})\subseteq  (\mathcal {C}_{\mathrm{ad}})'$, it follows from \cite[Proposition 2.5]{Mu} that 
\begin{equation}\label{G1}
s_{g,X} =\di_X  .
\end{equation}
Lastly, assume $X\in \mathcal{C}_\mathrm{pt}$. Since $g\in (\mathcal {C}_{\mathrm{ad}})_{\mathrm{pt}}$ and $\mathcal {C}_{\mathrm{pt}}\subseteq ((\mathcal {C}_{\mathrm{ad}})_{\mathrm{pt}})'$, again, by \cite[Proposition 2.5]{Mu}, we get 
\begin{equation}\label{G3}
s_{g,X} =\di_X.
\end{equation}
Now the result follows from equations \eqref{G2}, \eqref{G1} and \eqref{G3}.
\end{proof}

For every $X\in \mathcal {O(C)}$ consider the stabilizer subgroup $G[X]$ of $\mathcal{G(C)}$ given by $$G[X]=\{ g \in \mathcal{G(C)} :  g \otimes X \simeq X\}.$$ For any family $\mathcal F\subseteq \mathcal{O(C)}$, we define
\begin{equation*}
   G[\mathcal F] := \bigcap\limits_{X\in \mathcal F } G[X].
\end{equation*}
Note that $G[X]$ is a subgroup of $\mathcal{G}(\mathcal{C}_{\mathrm{ad}})$ for all simple $X \in \mathcal C$, and thus so is $G[\mathcal F]$. 

From the previous Lemma we get the following Corollary.

\begin{cor}\label{G} Suppose $(\mathcal{C}_{\mathrm{ad}})_{\mathrm{pt}}$ is isotropic and consider the family $$\mathcal F=\{X\in \mathcal{O(C)} \ | \  X\not\in \mathcal{G(C)}, \  X\not\in \mathcal{C}_{\mathrm{ad}}\}.$$
Then
\begin{enumerate}[leftmargin=*, label=(\alph*)]
\item The rows of the S-matrix that correspond to the objects of $G[\mathcal F]$ are all equal. 
\item If $\mathcal{C}$ is modular then $G[\mathcal F] =\{ 1\}$.
\end{enumerate}
\end{cor}


 \section{Results for classification of MTCs by dimension}\label{section: cd_p}

In this section, we study MTCs whose irreducible degrees are a multiple of an odd prime number $p$. In particular, we study MTCs that are integral and whose adjoint subcategory has dimension a power of $p$. We also need some more general auxiliary results that we prove in this section, see Proposition ~\ref{p2} and Lemma ~\ref{invertible_square free}.

\subsection{Non-pointed MTCs with irreducible degrees having a common factor $p$.}

In this subsection, let $\mathcal{C}$ be a non-pointed MTC and $p$ be an odd prime number such that  $$\cd(\mathcal C) \subset \{1\} \cup p \mathbb{Z}.$$

\begin{pro} \label{cyclic}
 If $\mathcal{G}(\mathcal{C}_{\mathrm{ad}})$ is a non-trivial p-group, then it is not cyclic. In particular, $|\mathcal{G}(\mathcal{C}_{\mathrm{ad}})| \geq p^2$.
\end{pro}

\begin{proof}
Assume $\mathcal{G}(\mathcal{C}_{\mathrm{ad}})$ is a cyclic group of order $p^k$, for some $k\in \mathbb{N}$. For every $i \in \{0,\dots, k\}$ denote by $H_i$ the unique subgroup of $\mathcal{G}(\mathcal{C}_{\mathrm{ad}})$ of order $p^i$. Note that
\begin{equation}\label{cyclic 1}
 H_0=\{e\} \subset H_1 \subset \dots \subset H_k.
\end{equation}
 We claim that there exists $1\leq l\leq k$ such that $H_l \subseteq G[X]$ for every non-invertible simple object $X\in \mathcal C$. As $p$ is odd, this is a contradiction by Corollary \ref{G}.  Indeed, taking the Frobenius-Perron dimension of both sides of 
\begin{equation*}
    X\otimes X^*= \bigoplus\limits_{g \in G[X]} g \oplus \bigoplus\limits_{Y \in \mathcal{C}_{\mathrm{ad}}, \di_Y>1} N_{XX^*}^Y Y,
\end{equation*}
  we obtain that $p$ divides $|G[X]|$. Thus $G[X]$ has order $p^{j_X}$, where $1\leq j_X \leq k$.
 
 Define $l := \min\limits_{X} j_X \geq 1$ and consider the subgroup $H_l$ of $\mathcal{G}(\mathcal{C}_{\mathrm{ad}})$. 
 By Equation \eqref{cyclic 1}, we have that $H_l \subseteq G[X]$ for every simple non-invertible $X$ in $\mathcal{C}$ and then the claim follows. 
\end{proof}
 
%


\begin{cor}\label{cyclic 2}
It follows that
\begin{enumerate}[leftmargin=*, label=(\alph*)]
\item \label{cyclic p-group} If $\mathcal{G}(\mathcal{C})$ is a cyclic p-group, then $\mathcal{G}(\mathcal{C}_{\mathrm{ad}})$ is trivial. 
\item \label{solvable} If $\mathcal{C}$ is solvable, then $\mathcal{G(C)}$ is not a cyclic p-group. 
\end{enumerate}
\end{cor}
\begin{proof}
\ref{cyclic p-group} If $\mathcal{G(C)}$ is a cyclic $p$-group, so is its subgroup $\mathcal{G}(\mathcal{C}_{\mathrm{ad}})$. By Proposition \ref{cyclic}, $\mathcal{G}(\mathcal{C}_{\mathrm{ad}})$ must be trivial.

\ref{solvable} This is a direct consequence of 
part \ref{cyclic p-group} since $\mathcal{C}_{\mathrm{ad}}$ contains a non-trivial invertible object \cite[Proposition 4.5]{ENO1}.
\end{proof}

\subsection{MTCs with dimension a power of $p$.}
In this work we are mostly interested in odd-dimensional categories, but some results hold also for $p = 2$. We do not assume that $p$ is odd unless otherwise stated. 

It is known that MTCs of dimension $p^n$ for $n\leq 4$ are pointed, see for example \cite{DN} for the result when $n=4$. We prove in this section that this is no longer true for $n=5$. For this, we need the following auxiliary proposition.

\begin{pro}\label{p2}
Let $\mathcal{C}$ be an integral MTC and $p$ be a prime number that divides $\FPdim(\mathcal{C})$. Then $\FPdim(\mathcal{C}_{\mathrm{ad}}) \ne p^2$.  
\end{pro}
\begin{proof}
Assume $\FPdim(\mathcal{C}_{\mathrm{ad}}) = p^2$. Then $\mathcal{C}$ is a non-pointed nilpotent category. Furthermore, $\mathcal{C}_{\mathrm{ad}}$ is solvable and thus it contains a non-trivial invertible object \cite[Proposition 4.5]{ENO2}. This together with the fact that every simple object has Frobenius-Perron dimension 1 or $p$ \cite[Corollary 5.3]{GN} implies that $\mathcal{C}_{\mathrm{ad}}$ is necessarily pointed. 

Note that by \cite[Corollary 2.7]{DGNO1} there is a non-trivial object $h \in \mathcal{C}_{\mathrm{ad}}$ such that  $\theta_h=1$. We show that the rows of the S-matrix corresponding to (the isomorphism classes of) $h$ and $\textbf{1}$ are equal. 

Given $g \in \mathcal{U(C)}$, let $a_g$ and $b_g$ denote the number of isomorphism classes of simple objects of dimension 1 and $p$ in the component $\mathcal{C}_g$, respectively. Since  $p^2 = \FPdim(\mathcal{C}_g) = a_g + b_g p^2$, the following holds:
\begin{itemize}
    \item Either $\mathcal{C}_g$ has exactly $p^2$ simple objects, all of which are invertible, or
    \item $\mathcal{C}_g$ has exactly one simple object, which has dimension $p$.
\end{itemize}

Let $X$ be a simple object of Frobenius-Perron dimension $p$ and let $g\in \mathcal U(C)$  such that $X \in \mathcal{C}_g$. As $X$ is the unique simple object in $\mathcal{C}_g$ we have $h \otimes X =X$. Using the balancing equation we obtain
\begin{equation}\label{6}
    s_{h,X}= \theta_h^{-1}\theta_X^{-1} \di_X \theta_X= \di_X.
\end{equation}

On the other hand, since $h\in \mathcal{C}_{\mathrm{ad}} =\mathcal{C}'_{\mathrm{pt}}$ it follows from \cite[Proposition 2.5]{Mu} that for all invertible object $g \in \mathcal{C}$ we have
\begin{equation}\label{5}
    s_{h,g}= \di_h\di_g= \di_g. 
\end{equation}

Equations \eqref{6} and \eqref{5} prove our claim.  This is a contradiction since $S$ is invertible. 
\end{proof}

Note that this yields an alternate proof of Lemma 4.11 of \cite{DN}.
\begin{cor}\label{p4}
Let $p$ be a prime. MTCs of Frobenius-Perron dimension $p^4$ are pointed.
\end{cor}

\begin{proof}
It is easy to see that $p^2$ divides $\FPdim(\mathcal{C}_\mathrm{pt})$. If $\mathcal{C}$ is not pointed we have $\FPdim({\mathcal C}_{\mathrm{pt}})=p^2$  \cite[Lemma 3.2]{DLD}. Thus the Frobenius-Perron dimension of $\mathcal C_{\mathrm{ad}}$ is also $p^2$, which cannot happen by Proposition \ref{p2}.
\end{proof}

Roughly speaking, $G$-gauging is a 2-step process to construct a new modular category from a given modular category with a categorical action of the group $G$ \cite{CGPW}. More specifically, first we extend the modular category to a $G$-crossed braided fusion category, following Etingof-Nikshych-Ostrik extension theory \cite{ENO3}, and then take the equivariantization of the resulting category.

\begin{theorem}\label{p5}Let $p$ be an odd prime. MTCs of Frobenius-Perron dimension $p^5$ are either pointed or a $\mathbb Z_p\times\mathbb Z_p$-gauging of a pointed MTC of Frobenius-Perron dimension $p$.
\end{theorem}

\begin{proof}
Let $\mathcal{C}$ be an integral MTC of Frobenius-Perron dimension $p^5$, and suppose that $\mathcal C$ is not pointed. Then by \cite[Lemma 3.2]{DLD} and Proposition \ref{p2} we get that $\FPdim(\mathcal{C}_{pt})=p^2$, 
and so $\mathcal C_{pt}\subseteq \mathcal C_{ad}$ \cite[Lemma 3.4]{DLD}. Hence $\mathcal C_{pt}$ is an odd-dimensional symmetric subcategory of $\mathcal C$ and thus $\mathcal C_{pt}\simeq \Rep(G)$ for some group $G$ of order $p^2$. Moreover, by Proposition \ref{cyclic}, we have that $G \simeq \mathbb Z_p \times \mathbb Z_p.$ 
	
Now, we consider the de-equivariantization $\mathcal C_G$ of $\mathcal C$ by the Tannakinan subcategory $\Rep(G).$ 
Let $\mathcal C_G^0$ denote the trivial component with respect to the associated $G$-grading. Since $\mathcal C$ is non-degenerate then $C_G^0$ is non-degenerate \cite[Proposition 4.6]{DGNO2}. 
Moreover, we have that $\FPdim(\mathcal C_G^0)=\FPdim(\mathcal C)/|G|^2=p$. 
Hence $\mathcal C$ is the gauging (see \cite{CGPW} for details on this construction) of a MTC $\mathcal D$ of dimension $p$ by the group $G\simeq \mathbb Z_p \times \mathbb Z_p$, as desired.
\end{proof}

\begin{lemma}\label{Cad}
Let $\mathcal C$ be a fusion category and $p$ a prime number such that $\FPdim(\mathcal{C})= p^3$. Then $\mathcal{C}_{\mathrm{ad}}$ is pointed.  
\end{lemma}
\begin{proof}
As $\mathcal{C}$ is nilpotent, $\FPdim(\mathcal{C}_{\mathrm{ad}})=1, p$ or $p^2$. If $X$ is a simple object in $\mathcal{C}_{\mathrm{ad}}$, then $\FPdim(X)$ can be either 1 or $p$ \cite[Corollary 5.3]{GN}. By a dimension argument, $\mathcal{C}_{\mathrm{ad}}$ must be pointed.
\end{proof}

\begin{lemma}\label{Cad2}
Let $\mathcal C$ be an integral MTC and $p$ an odd prime number such that $\FPdim(\mathcal{C}_{\mathrm{ad}})= p^3$. Then one of the following is true:
\begin{enumerate}
    \item $\mathcal{C}_{\mathrm{ad}}$ is pointed and  $\mathcal{G}(\mathcal{C}_{\mathrm{ad}})\simeq \mathbb{Z}_p \times \mathbb{Z}_p \times \mathbb{Z}_p.$
    \item  $(\mathcal{C}_{\mathrm{ad}})_{\mathrm{ad}}=(\mathcal{C}_{\mathrm{ad}})_{\mathrm{pt}}$ and   $\mathcal{G}(\mathcal{C}_{\mathrm{ad}})\simeq \mathbb{Z}_{p}\times \mathbb{Z}_p$.
\end{enumerate}
\end{lemma}

\begin{proof}
First note that the dimensions of simple objects of $\mathcal{C}$ are either $1$ or $p$ by \cite[Corollary 5.3]{GN}.

Assume $\mathcal{C}_{\mathrm{ad}}$ is not pointed. Then $\mathcal{C}_{\mathrm{ad}}$ has at least one simple object of dimension $p$. On the other hand, Lemma \ref{Cad} implies that $(\mathcal{C}_{\mathrm{ad}})_{\mathrm{ad}}$ must be pointed, and $\FPdim(X)^2$ divides $\FPdim((\mathcal{C}_{\mathrm{ad}})_{\mathrm{ad}})$ for all simple $X \in \mathcal{C}_{\mathrm{ad}}$ \cite[Corollary 5.3]{GN}. Therefore $\FPdim((\mathcal{C}_{\mathrm{ad}})_{\mathrm{ad}})=p^2$. Hence, by Proposition \ref{cyclic} we get $\mathcal{G}(\mathcal{C}_{\mathrm{ad}})\simeq\mathbb{Z}_{p}\times \mathbb{Z}_p$.

Now if $\mathcal{C}_{\mathrm{ad}}$ is pointed then $\mathcal{G}(\mathcal{C}_{\mathrm{ad}}) \simeq \mathbb{Z}_p \times \mathbb{Z}_p \times \mathbb{Z}_p$ or $\mathbb{Z}_{p^2} \times \mathbb{Z}_p$ by Proposition \ref{cyclic}. Assume the latter holds. Let $X$ be a simple object in $\mathcal{C}$ of dimension $p$. Since $X\otimes X^* \in \mathcal{C}_{\mathrm{ad}}$, then $X\otimes X^*= \bigoplus\limits_{g\in G[X]} \ g$. Hence, $|G[X]|=p^2$ for every non-invertible simple $X$. As the intersection of all subgroups of order $p^2$ of $\mathbb{Z}_{p^2}\times \mathbb{Z}_p$ is non-trivial, this is a contradiction by Proposition \ref{G}.
\end{proof}

\subsection{Weakly group-theoretical MTCs}

Recall that there is a major conjecture that states that every weakly integral fusion category is weakly group-theoretical~\cite[Question 2]{ENO2}.
The next proposition mimics the argument in the proof of Theorem 8.2 in \cite{N1}, which shows that odd-dimensional non-degenerate fusion categories with Frobenius-Perron dimension less than 33075 are solvable, and thus weakly group-theoretical. This bound follows from the result in \cite{N1} that MTCs with Frobenius-Perron dimension $p^aq^bc$, where $a$ and $b$ are nonnegative integers and $c$ is a square-free natural number, are also weakly
group-theoretical (this is a generalization of Bursnide's theorem for fusion categories, see \cite{ENO2}). We show another class of weakly integral categories that also turn out to be weakly group-theoretical, and further restricts the search for a (non-degenerate) counterexample to the conjecture.

\begin{pro}
Let $\mathcal C$ be a MTC of dimension $p^2q^2r^2$, where $p, q$ and $r$ are distinct odd prime numbers. Then $\mathcal C$ is weakly group-theoretical. 
\end{pro}

\begin{proof}
By the proof of \cite[Lemma 9.3]{ENO2} $\mathcal C$ contains a nontrivial symmetric subcategory $\mathcal D$. Since $\mathcal C$ is odd-dimensional then $\mathcal D$ is Tannakian and thus $\mathcal D \simeq \text{Rep}(G)$ for some finite group $G$. Since $\mathcal C$ is non-degenerate, so is its core $\mathcal C_G^0$ \cite{ENO1}. Moreover, $\dim(\mathcal C_G^0)=\dim(\mathcal C)/|G|^2$. Thus by \cite[Theorem 7.4]{N1} $\mathcal C_G^0$ is weakly group theoretical, and hence so is $\mathcal C.$
\end{proof}


\section{Perfect MTCs}\label{section: perfect} 

In analogy with group theory, we call a fusion category \emph{perfect} if the only $1$-dimensional simple object is the unit, that is, if $\mathcal{G(C)}= \{\textbf{1}\}$. Perfect fusion categories are sometimes called unpointed. Perfect fusion rings have been studied in \cite{LPW} and their character table in \cite[Theorem 41]{Bu}. Note that perfect fusion categories and pointed fusion categories are on opposite ends of the spectrum: in the former, the only pointed simple object is the unit, i.e. $\mathcal C_{\mathrm{pt}}=\text{Vec}$, and in the latter all simple objects are invertible, i.e. $\mathcal C_{\mathrm {pt}}=\mathcal C.$

\begin{pro}\label{nosymmetric}
Let $\mathcal{C}$ be a perfect odd-dimensional braided fusion category. Then $\mathcal{C}$ has no non-trivial symmetric subcategories. In particular, $\mathcal C$ is modular. 
\end{pro}

\begin{proof}
Let $\mathcal{E}$ be a symmetric subcategory of $\mathcal{C}$. As $\mathcal{C}$ is odd-dimensional, $\mathcal{E}$ is Tannakian and thus   $\mathcal{E} \simeq \Rep(G)$ for some finite group $G$. Note that the unit is the only invertible object in $\mathcal{C}$, hence $G$ must be a perfect group. Thus, $G$ must be trivial, as the order of every non-trivial finite perfect group is divisible by four and $\mathcal{C}$ is odd-dimensional. This implies that $\mathcal{E}$ is trivial.
\end{proof}

\begin{cor}\label{subcategoriesaremodular}
Let $\mathcal{C}$ be a perfect odd-dimensional MTC. Then every fusion subcategory of $\mathcal{C}$ is modular.  In particular,  $$\mathcal C \simeq \mathcal C\langle X \rangle \boxtimes \mathcal C\langle X \rangle'$$ for all $X\in \mathcal{O(C)}.$
\end{cor}

\begin{proof}
Let $\mathcal{D}$ be a non-trivial fusion subcategory of $\mathcal{C}$, and consider its M\"{u}ger centralizer $\mathcal{D}'$. Note that $\mathcal{D} \cap \mathcal{D}'$ is a symmetric subcategory of $\mathcal{C}$, and thus by Proposition \ref{nosymmetric} it must be trivial. That is, $\mathcal{D} \cap \mathcal{D}' \simeq \text{Vec}$, and therefore $\mathcal{D}$ is modular by \cite[Section 2.5]{DGNO1}.
\end{proof}
\begin{remark}
From Corollary \ref{subcategoriesaremodular}, it follows that it is enough to understand \emph{simple} perfect odd-dimensional MTCs, since any perfect odd-dimensional MTC is a Deligne product of simple ones. 
\end{remark}
\begin{cor}\label{perfect: nosubcategories}
Let $\mathcal{C}$ be a perfect odd-dimensional MTC of prime rank. Then $\mathcal{C}$ has no non-trivial fusion subcategories.
\end{cor}

\begin{pro}\label{coprime F-P}
Let $\mathcal{C}$ be an odd-dimensional perfect MTC 
and let $X, Y$ be simple objects in $\mathcal C$ with coprime Frobenius-Perron dimensions. Then 
  \begin{enumerate}[leftmargin=*, label=(\alph*)]
  \item \label{Sxy=0} Either $s_{X, Y} = 0$ or $s_{X, Y} = d_Xd_Y$. Moreover, $s_{X, Y} \not = 0$ if and only if $X$ and $Y$ centralize each other.
  \item\label{C prime then Sxy=0} If $\mathcal{C}$ is prime then $s_{X, Y} = 0$.
\end{enumerate}
\end{pro}
\begin{proof}
\ref{Sxy=0} Fix $X$ a simple object in $\mathcal C$. Consider the full subcategory $\mathcal{D}_{\lambda}^X$ of objects
$Z \in \mathcal C$ such that $\sigma_{Z,X}\sigma_{X,Z} = \lambda \cdot \id_{X\otimes Z}$ as in \cite[Lemma 3.15]{DGNO1}. Moreover, the category $\mathcal{D}^X = \oplus_{\lambda\in \textbf{k}^*} \mathcal{D}_{\lambda}^X$ is a fusion subcategory of $\mathcal{C}$. By \cite[Proposition 3.22]{DGNO1}, $\mathcal{D}^X = \langle Y\in \mathcal{O(C)} | Y \text{centralizes } X\otimes X^*\rangle = C_{\mathcal C}(\mathcal{C}\langle X \rangle_{\text{ad}})$.
By Corollary \ref{subcategoriesaremodular}, $\mathcal{C}\langle X \rangle$ is modular and, since $\mathcal{C}$ is perfect, so is $\mathcal{C}\langle X \rangle$. Therefore, $\mathcal{C}\langle X \rangle_{\text{ad}} = \mathcal{C}\langle X \rangle$ and $\mathcal{D}^X = C_{\mathcal C}((\mathcal{C}\langle X \rangle)_{\text{ad}}) = \mathcal{C}\langle X \rangle'$, by \cite[Corollary 6.9]{GN}. 

It follows from \cite[Lemma 7.1]{ENO2} that, since $\mathcal{D}^X = \mathcal{C}\langle X \rangle'$, for all $X$, $Y\in \mathcal{O(C)}$ of coprime Frobenius-Perron dimensions we have that $s_{X, Y} = 0$ or $s_{X, Y} = d_Xd_Y$ as desired.

\ref{C prime then Sxy=0} Consider the fusion subcategory $\mathcal C \langle X\rangle$. By Corollary \ref{subcategoriesaremodular}, we have a factorization $\mathcal C = \mathcal C \langle X \rangle \boxtimes \mathcal C \langle X\rangle '$. Since $\mathcal C$ is prime, we must have that $\mathcal C\langle X\rangle'$ is trivial, and thus $Y$ does not centralize $X$. Hence by part \ref{Sxy=0} we conclude $s_{X,Y}=0.$
\end{proof}
\begin{remark}
Part \ref{C prime then Sxy=0} of Proposition \ref{coprime F-P}  extends \cite[Lemma 10.2]{NPa} when the category is prime to the case $\gcd(\FPdim(X), \FPdim(Y)) = 2$. 
\end{remark}

We finish this section with the following result, which when applied to weakly integral perfect MTC shows that the Frobenius-Perron dimensions of all the simple objects in the category cannot have a common divisor. 

\begin{lemma}\label{invertible_square free}
Let $\mathcal{C}$ be a weakly-integral MTC such that $|\mathcal{G(C)}|$ is square-free. Then  $$\gcd\{\FPdim(X)  \ | X\in \mathcal{O(C)}, \  X\not\in \mathcal{G(C)} \}=1.$$
\end{lemma}

\begin{proof}
Assume there exists a prime $p$ such that $p$ divides $\dim(X)$ for all non-invertible simple $X$ in $\mathcal{C}$. By \cite[Theorem 2.11]{ENO2} we have that $p^2$ divides $\dim(\mathcal{C})$. Note that 
\begin{equation*}
    \FPdim(\mathcal{C})=|\mathcal{G(C)}| + \sum\limits_{X \in \mathcal{O(C)}\setminus \mathcal{G(C)}} \FPdim(X)^2,
\end{equation*}
and thus $p^2$ divides $|\mathcal{G(C)}|$ which is a square free number, a contradiction.
\end{proof}

\section{Results for classification of MTCs by rank}\label{section: ranks}

In this section we give lower bounds for the rank of a modular category in terms of data associated to its universal grading. We also use the action of $\mathcal{G(C)}$ on the graded components to gain information about their rank. We use this results on Section \ref{section: mnsd low rank} to prove our main Theorem.

\begin{lemma}\label{rank1}
Let $\mathcal{C}$ be a MTC and consider the universal grading $\mathcal{C}=\bigoplus\limits_{g\in \mathcal{U(C)}} \mathcal{C}_g.$ For every odd prime $p$ that divides $| \mathcal{G}(\mathcal{C}_{\mathrm{ad}})|$ we have the following:
\begin{enumerate}[leftmargin=*, label=(\alph*)]
    \item \label{item:rank-same-dimension} There exists $h\in \mathcal{U(C)}$, $h\ne 1$, such that $\mathcal{C}_h$ has at least $p$ non-invertible simple objects of the same dimension. 
    \item \label{item:rank-non-pointed} If $\mathcal{C}$ is non-pointed then  $$\rank(\mathcal{C}) \geq \rank(\mathcal{C}_{\mathrm{ad}}) + |\mathcal{G(C)}| + p -2.$$ 
   \item \label{item:rank-odd-dim} If $\mathcal{C}$ is odd-dimensional then
$$\rank(\mathcal{C}) \geq \rank(\mathcal{C}_{\mathrm{ad}}) + |\mathcal{G}(\mathcal C)| + 2p - 3 .$$ 
\end{enumerate}
\end{lemma}

\begin{proof}
Let $p$ be an odd prime that divides $|\mathcal{G}(\mathcal{C}_{\mathrm{ad}})|$.

\ref{item:rank-same-dimension} Let $g \in \mathcal{G}(\mathcal{C}_{\mathrm{ad}})$ of order $p$. Consider the fusion subcategory $\mathcal{C}\langle g\rangle$ generated by $g$. Note that $\mathcal{C}\langle g\rangle\subseteq (\mathcal C_{\mathrm{ad}})_{\mathrm{pt}}$ and thus $\mathcal{C}\langle g\rangle$ is a symmetric subcategory of $\mathcal{C}$ (see Remark \ref{Cad_pt_symm}). Since $\FPdim(\mathcal{C}\langle g\rangle)=p$ is odd then $\theta_g=1$ \cite[Corollary 2.7]{DGNO1}. 

For all $h\in \mathcal{G(C)}\setminus\{e\}$ consider the action of $g$ on the non-invertible simple objects of $\mathcal C_h$ given by left multiplication.  As the order of $g$ is $p$ it follows that for all $h\in \mathcal{G(C)}\setminus\{e\}$ this action is given either by the identity or by a cycle of length $p$.  If the former holds for all $h\in \mathcal{G(C)}\setminus\{e\}$, by Lemma \ref{gX=X} the rows of the $S$- matrix corresponding to $g$ and $\textbf{1}$ are equal, which is a contradiction as S is invertible. Thus, there must exist $h\ne 1$ such that $g$ acts as a cycle of length $p$ on the non-invertible simple objects of $\mathcal C_h$. Therefore there are at least $p$ different non-invertible simple objects in $\mathcal{C}_h$ of the same dimension. 

\ref{item:rank-non-pointed} Recall that all the components of the universal grading have at least one simple element. Hence $\rank(\mathcal{C}) \geq \rank(\mathcal{C}_{\mathrm{ad}}) + |\mathcal{U(C)}| + p -2 = \rank(\mathcal{C}_{\mathrm{ad}}) + |\mathcal{G(C)}| + p -2$.

\ref{item:rank-odd-dim} By part \ref{item:rank-same-dimension} there exists $h\ne \textbf 1$ such that $\mathcal{C}_h$ has at least $p$ non invertible simple objects. As $\mathcal{C}$ is odd-dimensional, by \cite[Corollary 8.2(ii)]{NS}, it is maximally non-self-dual. Thus, $\mathcal{C}_{h^{-1}}=\mathcal C_h^*$ also has at least $p$ non invertible simple objects, and the result follows.
\end{proof}

\begin{lemma}\label{rank 4}
Let $\mathcal{C}$ be a MTC such that $(\mathcal{C}_{\mathrm{ad}})_{\mathrm{pt}}$ is trivial. Then $\rank(\mathcal{C})= \rank(\mathcal{C}_{\mathrm{ad}})  \rank(\mathcal{C}_\mathrm{pt})$. 
\end{lemma}

\begin{proof}
Note that $\mathcal{C}_{\mathrm{ad}} \cap \mathcal{C}_{\mathrm{ad}}' \simeq \mathcal{C}_{\mathrm{ad}} \cap (\mathcal{C}_{\mathrm{ad}})_{\mathrm{pt}} \simeq \text{Vec}$. Thus $\mathcal{C}_{\mathrm{ad}}$ is modular and we have a braided equivalence $\mathcal{C} \simeq \mathcal{C}_{\mathrm{ad}} \boxtimes \mathcal{C}_\mathrm{pt}$. Therefore $\rank(\mathcal{C})= \rank(\mathcal{C}_{\mathrm{ad}})  \rank(\mathcal{C}_\mathrm{pt})$. 
\end{proof}

\begin{cor}\label{primerank}
Let $\mathcal{C}$ be a MTC of prime rank such that $ (\mathcal C_{\mathrm{ad}})_{\mathrm{pt}}$ is trivial. Then either $\mathcal{C}$ is pointed or $\mathcal{C}_\mathrm{pt}$ is trivial.\end{cor}


\begin{lemma} \label{dim cong rank}
Let $\mathcal{C}$ be an odd-dimensional fusion category. Then $\rank(\mathcal{D})\equiv \dim(\mathcal{D})$ $\modulo{8}$ for every fusion subcategory $\mathcal{D}$ of $\mathcal{C}$. 
\end{lemma}

\begin{proof}
Any odd integer $n$ satisfies $n^2\equiv 1 \mod 8$. As the dimension of every simple element in $\mathcal{C}$ is an odd integer we get 
\begin{equation*}
    \dim(\mathcal{D})= \sum\limits_{X \in \mathcal{O(D)}} \dim(X)^2 \equiv  \sum\limits_{X \in \mathcal{O(D)}} 1 = \rank(\mathcal{D})  \ \ \ \   \modulo{8}.
\end{equation*}
\end{proof}

\begin{remark}\label{cong 8}
Let $\mathcal{C}$ be an odd-dimensional MTC and consider the universal (faithful) grading $\mathcal{C}=\bigoplus\limits_{g \in \mathcal{U(C)}} \mathcal{C}_g$. As a direct consequence of Lemma ~\ref{dim cong rank}, We have 
\begin{align*}
    \rank(\mathcal{C}_{\mathrm{ad}}) \equiv \rank(\mathcal{C}_g) \modulo{8},
\end{align*}
 for all $g \in \mathcal{U(C)}.$
\end{remark}
\begin{pro}\label{rank Cg}
Let $\mathcal C$ be a non-pointed MTC such that $\mathcal{C}_\mathrm{pt}\subseteq \mathcal C_{\mathrm{ad}}$ and $\FPdim(\mathcal{C}_\mathrm{pt})=p^k$ for some odd prime $p$ and $k\in \mathbb N$. Consider the universal grading $\mathcal C=\bigoplus_{g\in \mathcal{G(C)}} \mathcal C_g$. Then 
\begin{align*}
    \rank(\mathcal C_g)\equiv  0 \modulo{p},
\end{align*}
for all $g\in \mathcal{G(C)}$ such that $g\ne 1$.
\end{pro}

\begin{proof}
Let $g\in \mathcal{G(C)}$ such that $g\ne 1$. Since $\mathcal{G(C)}$ acts on $\mathcal O(\mathcal C_g)$ by left multiplication and $\mathcal{G(C)}$ is a $p$-group, we have that the number of fixed objects by the action must be congruent to $|\mathcal O(\mathcal C_g)|$ modulo $p$. We show that there can be no fixed objects, and so the statement follows. 

Suppose there exists an element $X\in \mathcal O(\mathcal C_g)$ that is fixed by the action. Since $\mathcal{C}_\mathrm{pt}\subseteq \mathcal C_{\mathrm{ad}}$ and $\mathcal{C}_\mathrm{pt}'=\mathcal C_{\mathrm{ad}}$ we have that $\mathcal{C}_\mathrm{pt}$ is symmetric and odd-dimensional, and thus by \cite[Corollary 2.7]{DGNO1} we get $\theta_h=1$ for all $h\in \mathcal{C}_\mathrm{pt}.$ Hence by the balancing equation
\begin{align*}
    s_{h,X}=\theta_h^{-1}\theta_X^{-1}\theta_X \di_X=\di_X=\di_h\di_X, \text{ for all } h\in \mathcal{C}_\mathrm{pt}.
\end{align*}

Therefore $X\in \mathcal{C}_\mathrm{pt}'=\mathcal C_{\mathrm{ad}}$  \cite[Proposition 2.5]{Mu}, a contradiction.

\end{proof}

\begin{pro}\label{rank a}
Let $p$ be an odd prime and $a\in \mathbb C$. If $\mathcal C$ is a non-pointed MTC such that $\mathcal{C}_\mathrm{pt}\subseteq \mathcal C_{\mathrm{ad}}$ and $\FPdim(\mathcal{C}_\mathrm{pt})=p^k$ for some $k\in \mathbb N$, 
then  $\big| \{X\in \mathcal O(\mathcal C_g) \ |\  \dim(X)=a\}\big|$ is divisible by $p$ for all $g\in \mathcal{G(C)}$ such that $g\ne 1$.
\end{pro}

\begin{proof}
Let $g\in \mathcal{G(C)}$ such that $g\ne 1$ and let $a\in \mathbb C$. If  $\{X\in \mathcal O(\mathcal C_g) \ |\  \dim(X)=a\}$ is empty the statement is clear. Assume the set is not empty. Since $\mathcal{G(C)}$ acts on it by left multiplication and $\mathcal{G(C)}$ is a $p$-group, we have that the number of fixed objects by the action must be congruent to the cardinal of the set modulo $p$.  By the same argument given in the proof of Proposition \ref{rank Cg} there can be no fixed objects, and so the statement follows. 
\end{proof}

\section{Application: low rank odd-dimensional MTCs}\label{section: mnsd low rank}
This section is devoted to advancing the classification of odd-dimensional MTCs of low rank, giving a proof to our main Theorem \ref{thm:MNSD}.

Recall that odd-dimensional MTCs are MNSD \cite{NS}, and so they must have odd rank. Hence we restrict ourselves to odd-dimensional MTCs $\mathcal C$ such that $\rank(\mathcal{C})\in\{13, 15, 17, 19, 21, 23 \}$, since Bruillard and Rowell showed that all odd-dimensional MTCs of rank up to 11 are pointed, and found a non-pointed example of rank 25 \cite{BR}.

Note that proving Theorem \ref{thm:MNSD} is equivalent to showing that
\begin{enumerate}[leftmargin=*, label=(\alph*)]
\item if $\rank(\mathcal C)=13$ or $15$, then $|\mathcal{G(C)}|=\rank(\mathcal{C})$, 
\item if $\rank(\mathcal C)=17, 19, 21$ or $23$, then $|\mathcal{G(C)}|=\rank(\mathcal{C})$ or $1$.
\end{enumerate}

We will prove our statement discarding the different possibilities for $|\mathcal{G(C)}|$ until we are left with the cases stated above.

We start by proving the following useful Lemma. 
\begin{lemma}\label{Cadpt no trivial}
Let $\mathcal C$ be a non-pointed MNSD MTC such that $\rank(\mathcal{C})\in\{13, 15, 17, 19, 21, 23 \}$. Then $(\mathcal{C}_{\mathrm{ad}})_{\mathrm{pt}}$ is trivial if and only if $\mathcal{C}_\mathrm{pt}$ is trivial.
\end{lemma}
\begin{proof}
It follows from Corollary \ref{primerank} that $(\mathcal{C}_{\mathrm{ad}})_{\mathrm{pt}}$ trivial implies $\mathcal{C}_\mathrm{pt}$ trivial for ranks $13, 17, 19$ or $23$. 

Let $\rank(\mathcal{C})=15$ and assume $(\mathcal C_{\mathrm{ad}})_{\mathrm{pt}}$ is trivial. If $\mathcal{C}_\mathrm{pt}$ is not trivial, then by Corollary \ref{primerank} we have that $\mathcal{C}_{\mathrm{ad}}$ is a MNSD MTC of rank $3$ or $5$, and thus it is pointed  \cite{HR, RSW, BR}, which is a contradiction. 

Similarly, if $\rank(\mathcal{C})=21$ and we assume that $(\mathcal{C}_{\mathrm{ad}})_{\mathrm{pt}}$ is trivial but $\mathcal{C}_\mathrm{pt}$ is not trivial, then $\mathcal{C}_{\mathrm{ad}}$ is a MNSD modular of rank $3$ or $7$  \cite{HR, RSW, BR}, which is a contradiction. 
\end{proof}

\begin{remark}\label{casos descartados} Note that $|\mathcal{G(C)}|$ must be an odd integer smaller or equal to $\rank(\mathcal{C})$. By Lemma \ref{rank1} part \ref{item:rank-odd-dim} we must have that $\rank(\mathcal{C}) \geq \rank(\mathcal{C}_{\mathrm{ad}}) + |\mathcal{G}(\mathcal{C})| + 2p - 3 $ for all odd prime $p$ that divides $|\mathcal{G}(\mathcal{C}_{\mathrm{ad}})|$. From this and Lemma \ref{Cadpt no trivial} we conclude that the following are all the possible options for $|\mathcal{G(C)}|$:

\begin{enumerate}
    \item If $\rank(\mathcal{C})=13$, then $|\mathcal{G(C)}|= 3$ or $1$.
    \item If $\rank(\mathcal{C})=15, 17, 19, 21$ or $23$, then $|\mathcal{G(C)}|= 9, 5, 3$ or $1$.
\end{enumerate}
\end{remark}

We proceed discarding, case by case, all the possibilities stated above for $\rank(\mathcal{C})=13, 15$ and all possibilities stated above, besides $|\mathcal{G(C)}| = 1$, for $\rank(\mathcal{C})=17, 19, 21, 23$. 


\begin{theorem}\label{thm:MNSD}
Let $\mathcal C$ be an odd-dimensional MTC.
\begin{enumerate}[leftmargin=*, label=(\alph*)]
\item\label{item:MNSD-rank13,15} If $\rank(\mathcal C)=13$ or $15$, then $\mathcal C$ is pointed.
\item\label{item:MNSD-rank<25} If $\rank(\mathcal C)=17$, then $\mathcal C$ is either pointed, perfect, or has 3 invertible objects. Moreover, if the latter case exists, then $\mathcal C$ is equivalent to a categorification of the ring $R_{3,H}$  as defined in [JL, Definition 1.3], where $H$ is a finite abelian group of order 25 .
\item\label{item:MNSD-rank<25} If $\rank(\mathcal C)=17, 19, 21$ or $23$, then $\mathcal C$ is either pointed or perfect.
\end{enumerate}
\end{theorem} 

\begin{proof} Our proof will be given considering each rank as a separate case. The techniques used in each of the cases are similar, and we include the details for each of them for completeness.

\vspace{0.25cm}
\ref{item:MNSD-rank13,15} \textbf{Case rank($\mathcal C\textbf )=13$:} by Remark \eqref{casos descartados}, it is enough to discard the possibilities $|\mathcal{G(C)}|=3$ and $1$. Recall that by Lemma \ref{Cadpt no trivial}  $(\mathcal{C}_{\mathrm{ad}})_{\mathrm{pt}}$ is not trivial if $|\mathcal {G(C)}|\ne 1$. 

Assume $|\mathcal{G(C)}|= 3$. Note that $\mathcal{C}_\mathrm{pt} \subseteq \mathcal{C}_{\mathrm{ad}}$ and $\FPdim(\mathcal{C}_{\mathrm{ad}})$ cannot be equal to 3. Hence, there must exist a simple non-invertible element in $\mathcal{C}_{\mathrm{ad}}$, and as $\mathcal{C}$ is MNSD the rank of $\mathcal{C}_{\mathrm{ad}}$ is at least five. Moreover, by Lemma $\ref{rank1}$ part \ref{item:rank-same-dimension} there exists $g \in \mathcal{G(C)}\simeq \mathbb{Z}_3$ such that $ 3 \leq \rank(\mathcal{C}_g) = \rank(\mathcal{C}_{g^{-1}})$. Thus, $\mathcal{C}_{\mathrm{ad}}$ has rank either 5 or 7, and both cases are discarded by Remark \ref{cong 8}.
 
Assume now that $|\mathcal{G(C)}|=1$.
We will denote the non-invertible simple objects in $\mathcal{C}$ by $X_1, X_1^*, \cdots, X_6, X_6^*$, and their respective Frobenius-Perron dimensions by $\di_1, \cdots, \di_6$. Up to relabeling the simple objects, we have 
that $\di_1 \geq \di_2 \geq \cdots \geq \di_6$. Hence,
\begin{align}\label{h1}
    \dim(\mathcal{C})= 1 + 2 \di_1^2 + \cdots + 2 \di_6^2 \leq 1 + 12 \di_1^2.
\end{align}

On the other hand, by \cite[Theorem 2.11]{ENO2} there exists an odd integer $l$ such that $\dim(\mathcal{C}) = l \di_1^2$. Equation \eqref{h1} implies that $l\leq 12$, and therefore $l=5$ (see Lemma \ref{dim cong rank}).  Consequently, 

\begin{align}\label{h2}
    3\di_1^2= 1 + 2 \di_2^2 + \cdots + 2 \di_6^2 \leq 1 + 10 \di_2^2.
\end{align}

Again, by \cite[Theorem 2.11]{ENO2}, we know that $\di_2^2$ divides $\dim(\mathcal{C})= 5\di_1^2,$ and so there exists an odd integer $m$ such that $ \di_1^2 = m^2 \di_2^2$. Equation \eqref{h2} implies that $m=1$, that is, $\di_1= \di_2$ and  

\begin{align}\label{h3}
    \di_2^2= 1 + 2 \di_3^2 + \cdots + 2 \di_6^2 \leq 1 + 8 \di_3^2.
\end{align}

By the same argument as before, there exists an odd integer $n$  such that $\di_2^2= n^2 \di_3^2$, and Equation \eqref{h3} implies that $n=1$. Hence,

\begin{align*}
    \di_3^2= 1 + 2 \di_3^2 + \cdots + 2 \di_6^2 ,
\end{align*}

which is a contradiction.

\vspace{0.5 cm}
 \textbf{Case rank($\mathcal C\textbf )=15$:} by Remark \eqref{casos descartados} it is enough to discard the possibilities $|\mathcal{G(C)}|= 9, 5, 3$ and $1$. Recall that, by Lemma \ref{Cadpt no trivial}, $(\mathcal{C}_{\mathrm{ad}})_{\mathrm{pt}}$ is not trivial if $|\mathcal {G(C)}|\ne 1$.

Case $|\mathcal{G(C)}|= 9$: since $(\mathcal{C}_{\mathrm{ad}})_{\mathrm{pt}}$ is not trivial its Frobenius-Perron dimension must be at least 3. As $\FPdim(\mathcal{C}_{\mathrm{ad}})$ cannot be equal to 3, the rank of $\mathcal{C}_{\mathrm{ad}}$ is at least five. Thus, this case is discarded by Lemma \ref{rank1} part \ref{item:rank-odd-dim}, taking $p=3.$

Case $|\mathcal{G(C)}|= 5$: here $\mathcal{C}_\mathrm{pt} \subset \mathcal{C}_{\mathrm{ad}}$. As $\FPdim(\mathcal{C}_{\mathrm{ad}})$ cannot be equal to five, the rank of $\mathcal{C}_{\mathrm{ad}}$ is at least seven. Thus, this case is discarded by Lemma \ref{rank1} part \ref{item:rank-odd-dim}, taking $p=5.$

Case $|\mathcal{G(C)}|= 3$: again, $\mathcal{C}_\mathrm{pt} \subset \mathcal{C}_{\mathrm{ad}}$ and the rank of $\mathcal{C}_{\mathrm{ad}}$ must be at least five. Moreover, by Lemma $\ref{rank1}$ part \ref{item:rank-same-dimension} there exists $g \in \mathcal{G(C)}$ such that $ 3 \leq \rank({\mathcal C}_g) = \rank(\mathcal{C}_{g^{-1}})$. Thus, $\mathcal{C}_{\mathrm{ad}}$ has rank either $5, 7$ or $9$, and the last two are discarded by Remark \ref{cong 8}. Consider now the case where $\rank(\mathcal C_{\mathrm{ad}})=5$. By Remark \ref{cong 8}, we get that $\rank(\mathcal C_g)=5$ for all $g\ne 1$, which is a contradiction by Proposition \ref{rank Cg}. 
 
Lastly, assume $|\mathcal {G(C)}|=1$.
We will denote the simple non-invertible objects in $\mathcal{C}$ by $X_1, X_1^*, \cdots, X_7, X_7^*$ and their respective Frobenius-Perron dimensions by $\di_1, \cdots, \di_7$. Relabel the simple objects so that $\di_1 \geq \di_2 \geq \cdots \geq \di_7$. Hence,
\begin{align}\label{h4}
    \dim(\mathcal{C})= 1 + 2 \di_1^2 + \cdots + 2 \di_7^2 \leq 1 + 14 \di_1^2.
\end{align}

On the other hand, by \cite[Theorem 2.11]{ENO2}, there exists an odd integer $l$ such that $\dim(\mathcal{C}) = l \di_1^2$. Equation \eqref{h4} implies that $l\leq 14$, and therefore $l=7$ (see Lemma \ref{dim cong rank}).  Consequently, 

\begin{align}\label{h5}
    5\di_1^2= 1 + 2 \di_2^2 + \cdots + 2 \di_6^2 + 2 \di_7^2 \leq 1 + 12 \di_2^2.
\end{align}

Again, by \cite[Theorem 2.11]{ENO2}, we know that $\di_2^2$ divides $\dim(\mathcal{C})= 7\di_1^2$, and so there exists an odd integer $m$ such that $ \di_1^2 = m^2 \di_2^2$. Equation \ref{h5} implies that $m=1$, that is, $\di_1= \di_2$ and  

\begin{align}\label{h6}
    3 \di_2^2= 1 + 2 \di_3^2 + \cdots + 2 \di_6^2 + 2 \di_7^2\leq 1 + 10 \di_3^2.
\end{align}

By the same argument as before, there exists an odd integer $n$ 
such that $\di_2^2= n^2 \di_3^2$, and equation \eqref{h6} implies that $n=1$. Hence,

\begin{align}\label{h7}
    \di_3^2= 1 + 2 \di_4^2 + \cdots + 2 \di_6^2 + 2 \di_7^2\leq 1 + 8 \di_4^2,
\end{align}

Once again, there exists an odd integer $q$
such that $\di_3^2 = q^2 \di_4^2$, and equation \eqref{h7} implies $q=1$. Therefore, 

\begin{align*}
    \di_4^2= 1 + 2 \di_4^2 + \cdots + 2 \di_6^2 + 2 \di_7^2
\end{align*}
which is a contradiction.

\vspace{0.5 cm}
\ref{item:MNSD-rank<25} \textbf{Case rank($\mathcal C\textbf )=17$:}
by Remark \eqref{casos descartados}, it is enough to discard the cases  $|\mathcal{G(C)}|= 9,5$ and $3$. Recall that, by Lemma \ref{Cadpt no trivial},  $(\mathcal{C}_{\mathrm{ad}})_{\mathrm{pt}}$ is not trivial if $|\mathcal {G(C)}|\ne 1$. 
 
Case $|\mathcal{G(C)}|= 9$: Since $(\mathcal{C}_{\mathrm{ad}})_{\mathrm{pt}}$ is not trivial its Frobenius-Perron dimension must be at least 3. As $\FPdim(\mathcal{C}_{\mathrm{ad}})$ cannot be equal to 3, the rank of $\mathcal{C}_{\mathrm{ad}}$ is at least five. Thus, this case is discarded by Lemma \ref{rank1} part \ref{item:rank-odd-dim} taking $p=3$.

Case $|\mathcal{G(C)}|= 5$: here $\mathcal{C}_\mathrm{pt} \subset \mathcal{C}_{\mathrm{ad}}$. As $\FPdim(\mathcal{C}_{\mathrm{ad}})$ cannot be equal to five, the rank of $\mathcal{C}_{\mathrm{ad}}$ is at least seven. Thus, this case is discarded by Lemma \ref{rank1} part \ref{item:rank-odd-dim} taking $p=5$.

Case $|\mathcal{G(C)}|= 3$: as $\mathcal{C}_\mathrm{pt} \subset \mathcal{C}_{\mathrm{ad}}$, the rank of $\mathcal{C}_{\mathrm{ad}}$ is at least five. Moreover, by Lemma \ref{rank1} part \ref{item:rank-same-dimension} there exists $g \in \mathcal{G(C)}\simeq \mathbb{Z}_3$ such that $ 3 \leq \rank({\mathcal C}_g) = \rank(\mathcal{C}_{g^{-1}})$. Thus, $\mathcal{C}_{\mathrm{ad}}$ has rank either 5, 7 , 9 or 11. The first three  cases are discarded by Remark \ref{cong 8}. Assume $\rank(\mathcal{C}_{\mathrm{ad}})=11$.
We denote the non-invertible objects in $\mathcal{C}_{\mathrm{ad}}$ by $X_1, X_1^*, \cdots, X_4, X_4^*, $ and the invertible ones by $\textbf{1}, g, g^2$. Note that since $|\mathcal{G(C)}| = 3$, the action of $\mathcal{G(C)}$ by left multiplication on $\{X_1, X_1^*, \cdots, X_4, X_4^*\}$ has $2$ or $8$ fixed objects. 

Lets consider first the case in which there are exactly $2$ fixed objects by the action. That is, we have that (up to relabeling) the simple objects $X_1$ and $X_1^*$ are the only simple objects fixed by the action. Denote by $\di_i$ the Frobenius-Perron dimensions of the objects $X_i$ and $X_i^*$ for all $i$. It is easy to see that since $X_2, X_2^*, X_3, X_3^*, X_4, X_4^*$ are not fixed by the action, we have that $\di:=\di_{2}= \di_{3} = \di_{4}$. Thus, 
\begin{equation*}
    \dim(\mathcal{C})= \dim(\mathcal{C}_\mathrm{pt}) \dim(\mathcal{C}_{\mathrm{ad}}) = 9 + 6 \di_{1}^2 + 18 \di^2.
\end{equation*}

Hence, $\gcd(\di_{1}, \di)=1,3$. Assume $\gcd(\di_{1}, \di)=3$, and consider the decomposition 
$$X_2 \otimes X_2^* = \textbf{1} \oplus N_{X_2 X_2^*}^{X_1} X_1 \oplus \cdots \oplus N_{X_2X_2^*}^{X_4^*} X_4^*.$$
Taking dimensions on both sides, we get $$ \di^2= 1 +  \di_{1} (N_{X_2 X_2^*}^{X_1} + N_{X_2 X_2^*}^{X_1^*}) + \di (N_{X_2 X_2^*}^{X_2}+ \cdots+ N_{X_2 X_2^*}^{X_4^*}),$$
and thus 3 divides 1, which is a contradiction. Consequently, $\gcd(\di_{1}, \di) =1$. Let $ Y \in \{X_2, X_2^*, \cdots, X_4^*\}$. Consider the decomposition $$X_1 \otimes Y= N_{X_1 Y}^{X_1} X_1 \oplus \cdots \oplus N_{X_1 Y}^{X_4^*} X_4^*. $$
Notice that neither $g$ nor $g^2$ are subobjects of $X_1\otimes Y$ since $g$ fixes $X_1$ and $Y \not\simeq X_1^*$.
Taking dimensions on both sides on the previous equation we get
$$\di_{1}\di = \di_{1} (N_{X_1 Y}^{X_1} + N_{X_1 Y}^{X_1^*}) + \di (N_{X_1 Y}^{X_2}+ \cdots N_{X_1 Y}^{X_4^*}).$$

Thus, $\di_{1} $ divides $N_{X_1 Y}^{X_2}+ \cdots +N_{X_1 Y}^{X_4^*}$ and $\di$ divides $N_{X_1 Y}^{X_1} + N_{X_1 Y}^{X_1^*} $. Since  $\gcd(\di_1, \di)=1$, either $\di_{1} = N_{X_1 Y}^{X_2}+ \cdots + N_{X_1 Y}^{X_4^*}$ and $N_{X_1 Y}^{X_1} + N_{X_1 Y}^{X_1^*}=0$ or $\di = N_{X_1 Y}^{X_1} + N_{X_1 Y}^{X_1^*}$ and $N_{X_1 Y}^{X_2}+ \cdots + N_{X_1 Y}^{X_4^*}=0$. Assume the latter is true for some $Y\in \{X_2, \dots, X_4^*\}$. Then $N_{X_1 Y}^{X_2}= \cdots = N_{X_1 Y}^{X_4^*}=0$. In particular, $N_{X_1Y}^Y=0$ and so by the fusion rules we have that  $N_{Y Y^*}^{X_1}=0$. Thus 
$$Y \otimes Y^* = \textbf{1} \oplus N_{Y Y^*}^{X_2} X_2 \oplus\dots \oplus N_{Y Y^*}^{X_4^*} X_4^*,$$
and taking dimensions on both sides, we get
$$\di^2=1 + \di ( N_{Y Y^*}^{X_2}  +\dots + N_{Y Y^*}^{X_4^*}),$$
and so $\di$ divides 1, a contradiction. Therefore  $\di_{1} = N_{X_1 Y}^{X_2}+ \cdots + N_{X_1 Y}^{X_4^*}$ and $N_{X_1 Y}^{X_1} + N_{X_1 Y}^{X_1^*}=0$ for all $Y\in \{X_2, \dots, X_4^*\}$. Consequently, $N_{X_1 Y}^{X_1} = N_{X_1 Y}^{X_1^*}=0$ for all $Y\in \{X_2, \dots, X_4^*\}$, and so by the fusion rules we get that $N_{X_1 X_1^*}^{Y} = N_{X_1 X_1^*}^{Y^*}=0$ for all $Y\in \{X_2, \dots, X_4^*\}$.



Thus, 
$$X_1 \otimes X_1^* = \textbf{1} \oplus g \oplus g^2 \oplus N_{X_1 X_1^*}^{X_1} X_1 \oplus N_{X_1 X_1^*}^{X_1^*} X_1^*,$$
which implies that $\di_{1}=3$. Hence, $\dim(\mathcal{C})= 9 + 3^26 + 18  \di^2. $
By \cite[Theorem 2.11]{ENO2}, we get that $\di^2$ divides $3^27$. Thus, $\di^2= 9$, which is a contradiction as $\gcd(\di_{1}, \di)=1$ and $\di_1 = 3$.

Lastly, we consider the case in which all simple non-invertible objects in $\mathcal{C}_{\mathrm{ad}}$ are fixed by $\mathbb{Z}_3$. Recall that $\mathcal{C}= \mathcal{C}_{\mathrm{ad}} \oplus \mathcal{C}_g \oplus \mathcal{C}_{g^2}$, where $\rank(\mathcal{C}_{g})=\rank(\mathcal{C}_{g^2}) =3$. We denote the simple objects of $\mathcal{C}_{g}$ by $Y_1, Y_2, Y_3$ and their respective Frobenius-Perron dimensions by $\di_{Y_1}, \di_{Y_2}, \di_{Y_3}$. Note that the simple objects of $\mathcal{C}_{g}$ are exactly $Y_1^*, Y_2^*, Y_3^*$, and thus by Corollary \ref{G} the action of $\mathcal{G(C)}\simeq \mathbb{Z}_3$ by left multiplication on $\{Y_1, Y_2, Y_3\}$ must be non-trivial. We may relabel the simples so that $g \otimes Y_1 = Y_2$ and $g^2\otimes Y_1 = Y_3$. So $\di_{Y_1}=\di_{Y_2}=\di_{Y_3}=:\di$. Now, for all $i=1, \cdots, 4$, we have that $X_i\otimes Y_1 \in \mathcal{C}_{g}$, so 
\begin{align}\label{e1}
    X_i \otimes Y_1 = N_{X_i Y_1}^{Y_1} Y_1 \oplus N_{X_i Y_1}^{Y_2} Y_2 \oplus N_{X_i Y_1}^{Y_3} Y_3.
\end{align}
On the other hand, 
\begin{align}
    &X_i \otimes Y_1 = g \otimes X_i \otimes Y_1= N_{X_i Y_1}^{Y_1} Y_2 \oplus N_{X_i Y_1}^{Y_2} Y_3 \oplus N_{X_i Y_1}^{Y_3} Y_1,\label{e2}\\ 
     &X_i \otimes Y_1 = g^2 \otimes X_i \otimes Y_1= N_{X_i Y_1}^{Y_1} Y_3 \oplus N_{X_i Y_1}^{Y_2} Y_1 \oplus N_{X_i Y_1}^{Y_3} Y_2.\label{e3}
\end{align}

From equations \eqref{e1}, \eqref{e2}, \eqref{e3} we get that $N_{X_i Y_1}^{Y_1} =N_{X_i Y_1}^{Y_2}= N_{X_i Y_1}^{Y_3}  $, hence

\begin{align*}
    X_i \otimes Y_1 = N_{X_i Y_1}^{Y_1} (Y_1 \oplus Y_2 \oplus Y_3).
\end{align*}

Consequently, $\di_{X_i} = 3 N_{X_i Y_1}^{Y_1} $. So, 3 divides $\di_{X_i}$ for all $i=1,\cdots, 4$. Let $c_{X_i}= \di_{X_i}/3.$ 

Note that $ \dim(\mathcal{C})= 3 \dim(\mathcal{C}_{\mathrm{ad}})= 3 \dim(\mathcal{C}_g) = 9 \di^2.$  As $\di_{X_i}^2 $ divides $\dim(\mathcal{C}),$ we get $c_{X_i}^2$ divides $\di^2.$ Reordering the indices so that $c_{X_1} \geq c_{X_2} \geq c_{X_3} \geq c_{X_4},$ and letting $l$ be an odd integer such that $\di^2= l^2 c_{X_1}^2$, we get that  
\begin{align} \label{e4}
3 + 2 \di_{X_1}^2 + \cdots + 2 \di_{X_4}^2= \dim(\mathcal{C}_{\mathrm{ad}})= \dim(\mathcal{C}_g) = 3 \di^2.    
\end{align}
 Dividing each side of equation \eqref{e4} by 3, we get 
$$ l^2 c_{X_1}^2 =\di^2= 1+ 6 c_{X_1}^2 + \cdots + 6 c_{X_4}^2 \leq 1 + 24 c_{X_1}^2. $$

Hence, $l^2 \leq 25$, and so $l^2=1, 9$ or $25$. If $l^2=9,$ then 9 divides $\di^2$, and as $9$ also divides $\di_{X_1}^2, \cdots, \di_{X_4}^2$, by equation $\eqref{e4}$ we have that $9$ divides 3. Consequently, $l^2=1$, i.e, $\di^2 = c_{X_1}^2, $ and $\di_{X_1}^2 = 9 \di^2= \dim(\mathcal{C})= 3 \dim(\mathcal{C}_{\mathrm{ad}})= 9 + 6 \di_{X_1}^2 + \cdots+ 6 \di_{X_4}^2$, which is again a contradiction. 

Lastly, suppose that $l^2=25.$ Then $c_{X_1}=1$ and so $d_{X_i}=3$ for $1\leq i \leq 4.$ Then $d^2=25$ and so $d=5$. That is, $\mathcal C$ has 3 invertible objects and 6 objects of dimension 3 in $\mathcal{C}_{\mathrm{ad}}$, and 3 objects of dimension 5 in $\mathcal C_g$ for $g\ne e.$ Consider the action of $\mathbb Z_3\simeq \mathcal G(\mathcal C)$ in $\mathcal C$. Since all non-invertible simple objects in $\mathcal{C}_{\mathrm{ad}}$ are fixed by the action of $\mathbb Z_3$, the de-equivariantization $\mathcal C_{\mathbb Z_3}$ has 25 invertible objects. On the other hand, each non-trivial component $\mathcal C_g$ contains 3 simple objects that are permuted by the action of $\mathbb Z_3$ and thus induces an object of dimension 5 in $\mathcal C_{\mathbb Z_3}$. Hence $\mathcal C_{\mathbb Z_3}$ is a $\mathbb Z_3$-graded category of rank 27, with 25 invertible objects in its trivial component. Let $(\mathcal C_{\mathbb Z_5})_{\mathrm{pt}}\cong \operatorname{Vec}_H^{\omega}$ for some abelian group $H$ of order 25 and 3-cocycle $\omega$. Then $\mathcal K(\mathcal C_{\mathbb Z_3})\simeq R_{3,H}$, where  $\mathcal K(\mathcal C_{\mathbb Z_3})$  denotes the Grothendieck ring of $\mathcal C_{\mathbb Z_3}$ and $R_{3,H}$ is the fusion ring  defined on~\cite[Definition 1.3]{JL}.

\vspace{0.5 cm}
 \textbf{Case rank($\mathcal C\textbf )=19$:} by Remark \ref{casos descartados}, it is enough to discard the cases  $|\mathcal{G(C)}|= 9,5,$ and $3$. Recall that, by Lemma \ref{Cadpt no trivial}, $(\mathcal{C}_{\mathrm{ad}})_{\mathrm{pt}}$ is not trivial if $|\mathcal {G(C)}|\ne 1$. 

Case $|\mathcal{G(C)}|= 5$: here, $\mathcal{C}_\mathrm{pt} \subset \mathcal{C}_{\mathrm{ad}}$. As $\FPdim(\mathcal{C}_{\mathrm{ad}})$ cannot be equal to five, the rank of $\mathcal{C}_{\mathrm{ad}}$ is at least seven. Thus, this case is discarded by Remark \ref{cong 8}.
  
Case $|\mathcal{G(C)}|= 3$ or $9$: as $\FPdim(\mathcal{C}_{\mathrm{ad}})$ cannot be equal to 3, the rank of $\mathcal{C}_{\mathrm{ad}}$ is at least 5. By Lemma \ref{rank1} part \ref{item:rank-same-dimension} there exists $g \in \mathcal{G(C)}$ such that $ 3 \leq \rank(\mathcal{C}_g) = \rank(\mathcal{C}_{g^{-1}})$. Hence, $\rank(\mathcal{C}_{\mathrm{ad}})= 5, 7, 9, 11 $ or $13$, and all cases are discarded by Remark \ref{cong 8}.

\vspace{0.5 cm}
 \textbf{Case rank($\mathcal C\textbf )=21$:} by Remark \ref{casos descartados}, it is enough to discard the cases  $|\mathcal{G(C)}|= 15, 9, 5,$ and $3$. Recall that $(\mathcal{C}_{\mathrm{ad}})_{\mathrm{pt}}$ is not trivial if $|\mathcal {G(C)}|\ne 1$ by Lemma \ref{Cadpt no trivial}.  
 

 Case $|\mathcal{G(C)}|= 5$: here, $\mathcal{C}_\mathrm{pt}\subseteq \mathcal C_{\mathrm{ad}}$, and since $\FPdim(\mathcal{C}_{\mathrm{ad}})$ cannot be equal to 5 we get that $\rank(\mathcal{C}_{\mathrm{ad}})\geq 7$. By Lemma \ref{rank1} part \ref{item:rank-same-dimension} there exists $g \in \mathcal{G(C)}$ such that $ 5 \leq \rank(\mathcal{C}_g) = \rank(\mathcal{C}_{g^{-1}})$. Therefore $\rank(\mathcal{C}_{\mathrm{ad}})= 7$ or $9$, and both cases are discarded by Remark \ref{cong 8}.

Case $|\mathcal{G(C)}|= 9$: since $\FPdim(\mathcal{C}_{\mathrm{ad}})$ cannot be equal to 3 we get that $\rank(\mathcal{C}_{\mathrm{ad}})\geq 5$. By Lemma \ref{rank1} part \ref{item:rank-same-dimension} there exists $g \in \mathcal{G(C)}$ such that $ 3 \leq \rank(\mathcal{C}_g) = \rank(\mathcal{C}_{g^{-1}})$. Therefore $\rank(\mathcal{C}_{\mathrm{ad}})= 5, 7$ or $9$, and all cases are discarded by Remark \ref{cong 8}.

 Case $|\mathcal{G(C)}|= 3$: as $\FPdim(\mathcal{C}_{\mathrm{ad}})$ cannot be equal to 3 we have that $\rank(\mathcal{C}_{\mathrm{ad}})\geq 5$. By Lemma \ref{rank1} part \ref{item:rank-same-dimension} there exists $g \in \mathcal{G(C)}$ such that $ 3 \leq \rank(\mathcal{C}_g) = \rank(\mathcal{C}_{g^{-1}})$. Hence, $\rank(\mathcal{C}_{\mathrm{ad}})= 5, 7, 9, 11, 13$ or $15$, and all cases but $\rank(\mathcal{C}_{\mathrm{ad}})=  7$ are discarded by Remark \ref{cong 8}. If $\rank(\mathcal{C}_{\mathrm{ad}})=  7$ then $\rank(\mathcal C_g)=7$ for all $g\in \mathcal {G(C)}$, which is a contradiction by Proposition \ref{rank Cg}.

 \vspace{0.5 cm}
 \textbf{Case rank($\mathcal C\textbf )=23$:} by Remark \ref{casos descartados}, it is enough to discard the cases  $|\mathcal{G(C)}|= 15, 9, 5$, and $3$. Recall that by Lemma \ref{Cadpt no trivial}  $(\mathcal{C}_{\mathrm{ad}})_{\mathrm{pt}}$ is not trivial if $|\mathcal {G(C)}|\ne 1$.


Case $|\mathcal{G(C)}|= 9$:
since $\FPdim(\mathcal{C}_{\mathrm{ad}})$ cannot be equal to 3 we have that $\rank(\mathcal{C}_{\mathrm{ad}})\geq 5.$ By Lemma \ref{rank1} part \ref{item:rank-same-dimension} there exists $g \in \mathcal{G(C)}$ such that $ 3 \leq \rank(\mathcal{C}_g) = \rank(\mathcal{C}_{g^{-1}})$. Hence, $\rank(\mathcal{C}_{\mathrm{ad}})= 5, 7, 9$ or $11$, and all cases are discarded by Remark \ref{cong 8}. 

Case $|\mathcal{G(C)}|= 5$: since $\FPdim(\mathcal{C}_{\mathrm{ad}})$ cannot be equal to five we get that $\mathcal{C}_{\mathrm{ad}}\geq 7.$. By Lemma \ref{rank1} part \ref{item:rank-same-dimension} there exists $g \in \mathcal{G(C)}$ such that $ 5 \leq \rank(\mathcal{C}_g) = \rank(\mathcal{C}_{g^{-1}})$. Hence, $\rank(\mathcal{C}_{\mathrm{ad}})= 7, 9$ or $11$, and all cases are discarded by Remark \ref{cong 8}.

Case $|\mathcal{G(C)}|= 3$: since $\FPdim(\mathcal{C}_{\mathrm{ad}})$ cannot be equal to three, the rank of $\mathcal{C}_{\mathrm{ad}}$ is at least five. By Lemma \ref{rank1} part \ref{item:rank-same-dimension} there exists $g \in \mathcal{G(C)}$ such that $ 3 \leq \rank(\mathcal{C}_g) = \rank(\mathcal{C}_{g^{-1}})$. Hence, $\rank(\mathcal{C}_{\mathrm{ad}})= 5, 7, 9, 11, 13, 15$, and all cases but $\rank(\mathcal{C}_{\mathrm{ad}})=  13$ are discarded by Remark \ref{cong 8}. Now, if  $\rank(\mathcal{C}_{\mathrm{ad}})=  13$ then $\rank(\mathcal{C}_g)=5$ for $g\in \mathcal {G(C)}$ such that $g\ne 1$, which is a contradiction by Proposition \ref{rank Cg}.
\end{proof}

\bibliographystyle{plain}

\end{document}